\documentclass[10pt, reqno]{amsart}
\usepackage{fancyhdr}
\usepackage{a4wide, amsthm, amscd, amsfonts, amssymb, graphicx,tikz, color, environ}
\usepackage{amsaddr}
\usepackage{amsmath,amssymb,amsfonts,amsthm, graphicx,color,fancyhdr}
\usepackage[latin1]{inputenc}
\usepackage{enumerate}

\usepackage{psfrag}
\usepackage[colorlinks=true, pdfstartview=FitV, allcolors=black]{hyperref}
\usepackage{psfrag,caption}
\usepackage[all]{xy}
\usepackage{comment}
\usepackage{appendix}
\usepackage{verbatim}
\usepackage{mathrsfs}
\usepackage{pdfsync}
\usepackage{graphicx,epsfig}
\newtheorem{thm}{Theorem}[section]

\theoremstyle{stylename}

\let\oldproofname=\proofname
\renewcommand{\proofname}{\rm\bf{\oldproofname}}

 \newtheorem{lemma}[thm]{Lemma}

\theoremstyle{definition}
 
 \newtheorem{rmk}[thm]{Remark}

\subjclass[2010]{Primary  35P15; Secondary  58J50}
\begin{document}

\title[An isoperimetric inequality for Steklov eigenvalues]
 {An isoperimetric inequality for harmonic mean of Steklov eigenvalues in Hyperbolic space}
\author{Sheela Verma}
\address{Tata Institute of Fundamental Research \\ Centre For Applicable Mathematics \\ Bangalore, India}
\email{sheela.verma23@gmail.com}

%\thanks{This work was supported by the project PDA/IITK/MATH/96062.}
%----------Author 2

\begin{abstract}
 In this article, we prove an isoperimetric inequality for the harmonic mean of the first $(n-1)$ nonzero Steklov eigenvalues on bounded domains in $n$-dimensional Hyperbolic space. Our approach to prove this result also gives a similar inequality for the first $n$ nonzero Steklov eigenvalues on bounded domains in $n$-dimensional Euclidean space.
\end{abstract}
\keywords{Isoperimetric inequality, Steklov eigenvalue problem, Exponential map, geodesic normal coordinate system}
%%% ----------------------------------------------------------------------
\maketitle
%%% ----------------------------------------------------------------------
%\tableofcontents
\section{Introduction}
Let $\Omega$ be a bounded domain in a complete Riemannian manifold $(M, ds^{2})$ with smooth boundary $\partial \Omega$. Consider the Steklov eigenvalue problem on $\Omega$
\begin{align*}
\begin{array}{rcll}
\Delta u &=& 0 & \mbox{ in } \Omega ,\\
\frac{\partial u}{\partial \nu} &=& \mu u  &\mbox{ on } \partial \Omega.
\end{array}
\end{align*}
Here $\Delta := - \mbox{div}(\mbox{grad } u)$, $\nu$ is the outward unit normal to $\partial \Omega$ and $\frac{\partial u}{\partial \nu}$ denotes the directional derivative of $u$ in the direction $\nu$.

This problem was introduced by Steklov \cite{S} for bounded domains in the plane in $1902$. Its importance lies in the fact that the set of eigenvalues of the Steklov problem is same as the set of eigenvalues of the well known Dirichlet-Neumann map. This map associates each function $u$ defined on $\partial \Omega$ to the normal derivative of its harmonic extension on $\Omega$. It is known that the Steklov eigenvalues are discrete and forms an increasing sequence $0 = \mu_{0}(\Omega) < \mu_{1}(\Omega) \leq \mu_{2}(\Omega) \leq \cdots \nearrow \infty$.

The interplay between the geometry of manifold and the Steklov eigenvalues has recently attracted substantial attention. See \cite{CGG,GLS,GP,PS,SV,SVGS} and the references therein for recent development. The problem of finding a domain under some geometric constraints, which optimizes eigenvalues (or some combination of eigenvalues) is a classical question in spectral geometry. In this direction, for Steklov eigenvalue, the first result was given by Weinstock \cite{W} in $1954$. Using conformal map technique, he proved that among all simply connected planar domains with analytic boundary of fixed perimeter, the circle maximizes $\mu_{1}$. Hersch and Payne \cite{HP} noticed that Weinstock's proof gives a sharper isoperimetric inequality
\begin{align} \label{result HP}
\frac{1}{\mu_{1}(\Omega)} + \frac{1}{\mu_{2}(\Omega)} \geq \frac{P(\Omega)}{\pi},
\end{align}
where $P(\Omega)$ represents the perimeter of $\Omega \subset \mathbb{R}^{2}$.

Later F. Brock \cite{B} generalized result eqref{result HP} to $\mathbb{R}^{n}$ by fixing the volume of the domain and proved the following inequality: for a bounded Lipschitz domain $\Omega \subset \mathbb{R}^{n}$,
\begin{align} \label{ineq: FB}
\sum_{i=1}^{n} \frac{1}{\mu_{i}(\Omega)} \geq \frac{n}{\mu_{1}(B(R)},
\end{align}
where $B(R) \subset \mathbb{R}^{n}$ is a ball of radius $R$ such that vol$(\Omega))$ = vol$(B(R))$.

In this paper, we extend Brock's result and prove an isoperimetric inequality for the sums of reciprocals of the first $(n-1)$ nonzero steklov eigenvalues on bounded domains in $n$-dimensional Hyperbolic space. Further, our technique also provides a geometric proof of inequality \eqref{ineq: FB}. To the best of our knowledge, this is the first attempt to study the harmonic mean of Steklov eigenvalues on bounded domains in hyperbolic space. The main results of this article are as follows.

\begin{thm} \label{thm: hyperbolic space}
Let $\mathbb{H}^{n}$ be an $n$-dimensional hyperbolic space with constant curvature $-1$ and $\Omega \subset \mathbb{H}^{n}$ be a bounded domain with smooth boundary $\partial\Omega$. Then
\begin{align*}
 \sum_{i=1}^{n-1} \frac{1}{\mu_{i}(\Omega)} \geq \sum_{i=1}^{n-1} \frac{1}{\mu_{i}(B(R))},
\end{align*}
where $B(R) \subset \mathbb{H}^{n}$ be a geodesic ball of radius $R>0$ such that vol$(\Omega)$ = vol$(B(R))$. Further, equality holds if and only if $\Omega$ is a geodesic ball.
\end{thm}

\begin{thm} \label{thm: Euclidean space}
Let $\Omega \subset \mathbb{R}^{n}$ be a bounded domain with smooth boundary $\partial\Omega$ and $B(R) \subset \mathbb{R}^{n}$ be a geodesic ball of radius $R>0$ such that vol$(\Omega)$ = vol$(B(R))$. Then
\begin{align*}
 \sum_{i=1}^{n} \frac{1}{\mu_{i}(\Omega)} \geq n R = \sum_{i=1}^{n} \frac{1}{\mu_{i}(B(R))}.
\end{align*}
Further, equality holds if and only if $\Omega$ is a geodesic ball.
\end{thm}

The rest of the paper is organized as follows. In Section $2$, we state several known facts related to first nonzero Steklov eigenvalue of geodesic ball and also prove some results which are used to prove the main theorems. Proof of Theorems \ref{thm: hyperbolic space} and \ref{thm: Euclidean space} is given in Section $3$.
\section{Preliminaries}
In this section, we begin with providing some facts related to first $n$ nonzero Steklov eigenvalues on geodesic ball in $\mathbb{H}^{n} \mbox{ and } \mathbb{R}^{n}$ and then prove some properties of its eigenfunction.

Let $M = \mathbb{H}^{n} \mbox{ or } \mathbb{R}^{n}$ with the Riemannian metric $ds^{2} = dr^{2} + \sin_{k}^{2} r g_{\mathbb{S}^{n-1}}$, where $g_{\mathbb{S}^{n-1}}$ represents the canonical metric on the $(n-1)$- dimensional unit sphere $\mathbb{S}^{n-1}$ and
\begin{align*}
\sin_{k} r =
\begin{cases}
r, & \mbox{ if }M = \mathbb{R}^{n} \\
\sinh r, & \mbox{ if }M = \mathbb{H}^{n}.
\end{cases}
\end{align*}

Define
\begin{align*}
\cos_{k} r =
\begin{cases}
1, & \mbox{ if }M = \mathbb{R}^{n} \\
\cosh r, & \mbox{ if }M = \mathbb{H}^{n}.
\end{cases}
\end{align*}

%\subsection{Geodesic normal coordinate system}
%Let $\{e_{i}\}_{i=1}^{n}$ be an orthonormal basis of $\mathbb{R}^{n}$. Fix $p \in M$. Then for any point $q \in M$, there exists $v \in U_{p}M$ such that $q = \exp_{p}(v)$, where $\exp_{p}$ is the exponential map defined on $T_{p}M$ and
%\begin{align*}
%U_{p}M = \{ v \in T_{p}M : ds^{2}(v,v)=1 \}.
%\end{align*}

\subsection{Properties of first $n$ nonzero Steklov eigenvalues on Ball}
Let $B(R)$ be a geodesic ball in $(M, ds^{2})$, then $\mu_{1}(B(R)) = \mu_{2}(B(R)) = \cdots = \mu_{n}(B(R))$ and $\mu_{i}(B(R))$, $1 \leq i \leq n$ satisfies
\begin{align*}
\mu_{i}(B(R))= \frac{\int_{B(R)} g^{2}(r) \frac{n-1}{\sin_{k}^{2} r} + (g'(r))^{2} dV}{g^{2}(R) \mbox{vol}(S(R))},
\end{align*}
where $g(r)$ is the radial function satisfying
\begin{align} \label{eqn function g}
\begin{array}{rcll}
g''(r) + \frac{(n-1)\cos_{k} r}{\sin_{k} r} g'(r) - \frac{(n-1)}{\sin_{k}^{2} r}g(r) = 0, r \in (0,R) \\
g(0) = 0, \quad g'(R) = \mu_{1}(B(R)) g(R).
\end{array}
\end{align}
The equation $g''(r) + \frac{(n-1)\cos_{k} r}{\sin_{k} r} g'(r) - \frac{(n-1)}{\sin_{k}^{2} r}g(r) = 0$ can be rewritten as
\begin{align*}
g''(r) + \left(\frac{(n-1)\cos_{k} r}{\sin_{k} r} g(r)\right)'  = 0.
\end{align*}
Simplifying the above equation, we get
\begin{align} \label{function g}
g(r) = \frac{1}{\sin_{k}^{n-1} r} \int_{0}^{r}\sin_{k}^{n-1} t dt.
\end{align}
See \cite{BG} for further details.
\begin{rmk}
In case of $\mathbb{R}^{n}$, $g(r) = \frac{r}{n}$ and $\mu_{1}(B(R)) = \frac{1}{R}$.
\end{rmk}

The function $g$ defined in \eqref{function g} satisfies the following properties:
\begin{lemma} \label{lem: property g}
For $r > 0$,
\begin{enumerate} [(i)]
\item $0 < g'(r) \leq \frac{g(r)}{\sin_{k} r}$,
\item $(g')^{2}(r) + \frac{n-1}{\sin_{k}^{2} r} g^{2}(r)$ is a decreasing function of $r$.
\end{enumerate}
\end{lemma}
\begin{proof}
\begin{enumerate}[(i)]
\item It can be seen easily that in case of $M = \mathbb{R}^{n}$ it holds true. For $M = \mathbb{H}^{n}$: Since $g'(r) = 1 - (n-1) g(r) \coth r  $, it follows that $g'(r) \leq \frac{g(r)}{\sinh r}$ if and only if
\begin{align*}
(1 + (n-1) \cosh r) \left(\int_{0}^{r}\sinh^{n-1} t dt\right) \geq \sinh^{n} r.
\end{align*}
Let $h_{1}(r) = (1 + (n-1) \cosh r) (\int_{0}^{r}\sinh^{n-1} t dt) $ and $h_{2}(r) = \sinh^{n} r$. Since $h_{1}(0) = h_{2}(0) = 0$ and $h_{1}(r), h_{2}(r)$ both are increasing function of $r$, to show that $h_{1}(r) \geq h_{2}(r)$, it is enough to prove that ${h_{1}}'(r) \geq {h_{2}}'(r)$ for $r > 0$. Next ${h_{1}}'(r) \geq {h_{2}}'(r)$ if and only if
\begin{align*}
 (n-1) \int_{0}^{r}\sinh^{n-1} t \ dt \geq (\cosh r - 1) \sinh^{n-2} r.
\end{align*}
Using similar argument we have that the above inequality is true if
\begin{align*}
 \sinh^{n-1} r \geq (\cosh r - 1) \sinh^{n-3} r \cosh r.
\end{align*}
This can be rewritten as $\cosh r \geq 1$. which is true. Hence $g'(r) \leq \frac{g(r)}{\sin_{k} r}$. In similar way, it can be proved that $g'(r) > 0$ for $r > 0$.
\item This can be proved by using a similar argument as in $(i)$. We refer to \cite{BG} for a complete proof.
\end{enumerate}
\end{proof}

Now using the Borsuk-Ulam theorem, we construct test functions for the first $n$ nonzero Steklov eigenvalues on bounded domain $\Omega \subset M$.
\subsection{Construction of test functions}
Let $A$ be a domain in $M$ and $CA$ denotes the convex hull of $A$. Let $\exp_{p} : T_{p}M \rightarrow M$ be the exponential map and $(X_{1}, X_{2}, \ldots, X_{n})$ be a geodesic normal coordinate system at $p$. We identify $CA$ with $\exp_{p}^{-1}(CA)$ and $ds_{p}^{2}(X,X)$ as $\|X\|_{p}^{2}$ for $X \in T_{p}(M)$. The following lemma gives the existence of a center of mass of $A$ in $M$ (see \cite{GS}).

\begin{lemma} \label{lem:center of mass}
Let $G: [0, \infty)\rightarrow \mathbb{R}$ be a continuous function which is positive on $(0, \infty)$. Then there exists a point $q \in CA$ such that
\begin{align*}
\int_{A} G(\|X\|_{q}) X ds = 0,
\end{align*}
where $(X_{1}, X_{2}, \ldots, X_{n})$ is a geodesic normal coordinate system at $q$.
\end{lemma}

The above lemma will be used to construct a test function for the first nonzero Steklov eigenvalue on bounded domains in $M$. Let $\Omega$ be a bounded domain in $M$ and $B(R)$ be a ball in $M$ such that vol$(\Omega)$=vol$(B(R))$. We denote by $\{u_{i}\}_{i=0}^{\infty}$, a sequence of orthonormal set of eigenfunctions corresponding to eigenvalues $\{\mu_{i}(\Omega)\}_{i=0}^{\infty}$. The variational characterization of $\mu_i(\Omega)$, $1 \leq i < \infty$ is given by
\begin{align} \label{general variational characterization}
\mu_{i}(\Omega)= \inf_{0 \neq u \in H^{1}(\Omega)} \left\{ \frac{\int_\Omega{\|\nabla u\|^2}\, dV}{\int_{\partial\Omega}{u^2}\, dA} : \int_{\partial\Omega} u u_{j} dA = 0, 0 \leq j \leq {i-1} \right\}.
\end{align}
 Then by Lemma \ref{lem:center of mass}, there exists a point $p \in C(\partial \Omega)$ such that
\begin{align*}
\int_{\partial \Omega} g(r) \frac {x_{i}}{r} dA = 0, \quad 1 \leq i \leq n.
\end{align*}
Here $g(r)$ is the function defined in \eqref{function g}, $(x_{1}, x_{2}, \ldots, x_{n})$ is a geodesic normal coordinate system at $p$ and $r = \|\exp_{p}^{-1}(q)\|_{p}$, $q \in \partial \Omega$.

Next define mapping $f_{n}: \mathbb{S}^{n-1}\rightarrow \mathbb{R}^{n-1}$ as
\begin{align*}
f_{n}(\sigma) = \left(\int_{\partial \Omega} g(r) \frac {\langle\exp_{p}^{-1}(q), \sigma \rangle}{r} u_{1} dA, \int_{\partial \Omega} g(r) \frac {\langle\exp_{p}^{-1}(q), \sigma \rangle}{r} u_{2} dA, \ldots, \int_{\partial \Omega} g(r) \frac {\langle\exp_{p}^{-1}(q), \sigma \rangle}{r} u_{n-1} dA\right)
\end{align*}
 Since the function $f_{n}$ is antipode preserving, by the Borsuk-Ulam theorem, there exists $\sigma_{n} \in \mathbb{S}^{n-1}$ such that $f_{n}(\sigma_{n})=0$. Then we repeat the similar process with functions $f_{n-1}, f_{n-2}, \ldots, f_{2}$ in sequence, where the function $f_{m}: \mathbb{S}^{m-1}\rightarrow \mathbb{R}^{m-1}$, $2 \leq m \leq (n-1)$ is defined componentwise as
\begin{align*}
f_{m,k}(\sigma) = \int_{\partial \Omega} g(r) \frac {\langle\exp_{p}^{-1}(q), \sigma \rangle}{r} u_{k} dA, 1 \leq k \leq (m-1).
\end{align*}
Here $\sigma \in \mathbb{S}^{m-1} $ ranges over all unit vectors perpendicular to the ones already fixed. In particular, $f_{n-1}$ acts on $\sigma \in \mathbb{S}^{n-2}$, which is perpendicular to already fixed unit vector $\sigma_{n}$ in the previous step. By repeating this inductive process, once $\sigma_{2}$ is chosen, $\sigma_{1}$ is effectively fixed.

Next we define a new normal coordinate system at $p$ by considering $\sigma_{1}, \sigma_{2}, \ldots, \sigma_{n}$ as $x_1$-axis, $x_2$-axis, \ldots, $x_n$-axis respectively. Since new coordinate system is obtained by performing rotations, we have $\int_{\partial \Omega} g(r) \frac {x_{i}}{r} dA = 0, 1 \leq i \leq n$ in this coordinate system. Also in this new coordinate system
\begin{align} \label{eqn:test function}
\int_{\partial\Omega} g(r) \frac {x_{i}}{r} u_{j} dA = 0, \quad 0 \leq j \leq {i-1}.
\end{align}

\section{Proof of main results}
Let $\Omega$ be a bounded domain in $M$ with smooth boundary and $p$ be the center of mass defined as in the above section. Let $B(R)$ be the geodesic ball of radius $R>0$ in $M$ centered at $p$ such that vol$(\Omega)$ = vol$(B(R))$. We denote the geodesic sphere of radius $R>0$ in $M$ centered at $p$ by $S(R)$.
The following lemma is needed to prove main result (see \cite{BG}).
\begin{lemma} \label{lem: integral g}
 Let $g$ be the function defined as in \eqref{function g}. Then
\begin{align*}
 \int_{\partial\Omega} g^{2}(r) dA \geq \mbox{vol} \left(S(R)\right) g^{2}(R),
 \end{align*}
 and equality holds if and only if $\partial\Omega$ is a geodesic sphere of radius $R$ centered at $p$.
\end{lemma}

From the variational characterization of $\mu_{i}(\Omega)$, $1 \leq i \leq n$ and \eqref{eqn:test function}, we have
\begin{align*}
\mu_{i}(\Omega) \int_{\partial\Omega} \left(g(r)\frac{x_{i}}{r}\right)^{2} dA \leq \int_{\Omega} \Big\|\nabla \left(g(r)\frac{x_{i}}{r}\right)\Big\|^{2} dV.
\end{align*}
Substituting $\|\nabla \left(g(r)\frac{x_{i}}{r}\right)\|^{2} = (g'(r))^{2} \frac{x_{i}^{2}}{r^{2}} + \frac{g^{2}(r)}{\sin_{k}^{2} r} \left(1 - \frac{x_{i}^{2}}{r^{2}}\right)$ in the above equation, we get
\begin{align*}
 \int_{\partial\Omega} \left(g(r)\frac{x_{i}}{r}\right)^{2} dA \leq \frac{1}{\mu_{i}(\Omega)} \int_{\Omega} \left((g'(r))^{2} \frac{x_{i}^{2}}{r^{2}} + \frac{g^{2}(r)}{\sin_{k}^{2} r} \left(1 - \frac{x_{i}^{2}}{r^{2}}\right) \right)dV.
\end{align*}
Summing over $i$ from $1$ to $n$ in the above equation, we get
\begin{align}\label{eqn: sum1}
 \int_{\partial\Omega} g^{2}(r) dA \leq \sum_{i=1}^{n} \frac{1}{\mu_{i}(\Omega)} \int_{\Omega} \frac{g^{2}(r)}{\sin_{k}^{2} r} dV+ \sum_{i=1}^{n} \frac{1}{\mu_{i}(\Omega)} \int_{\Omega}\left((g'(r))^{2} - \frac{g^{2}(r)}{\sin_{k}^{2} r}\right)  \frac{x_{i}^{2}}{r^{2}} dV.
\end{align}

\subsection{Proof of Theorem \ref{thm: hyperbolic space}}
Using the fact that $\sum_{i=1}^{n} x_{i}^{2} = r^{2}$ , we have
\begin{align} \nonumber
\sum_{i=1}^{n} \frac{1}{\mu_{i}(\Omega)} \frac{x_{i}^{2}}{r^{2}} &= \sum_{i=1}^{n-1} \frac{1}{\mu_{i}(\Omega)} \frac{x_{i}^{2}}{r^{2}} + \frac{1}{\mu_{n}(\Omega)} \frac{x_{n}^{2}}{r^{2}} \\ \nonumber
& = \sum_{i=1}^{n-1} \frac{1}{\mu_{i}(\Omega)} \frac{x_{i}^{2}}{r^{2}} + \frac{1}{\mu_{n}(\Omega)}\left(1 - \sum_{i=1}^{n-1} \frac{x_{i}^{2}}{r^{2}} \right)\\ \label{eqn: hyperbolic 1}
& = \sum_{i=1}^{n-1} \left(\frac{1}{\mu_{i}(\Omega)} - \frac{1}{\mu_{n}(\Omega)} \right) \frac{x_{i}^{2}}{r^{2}} + \frac{1}{\mu_{n}(\Omega)}.
\end{align}
Since $0 < g'(r) \leq \frac{g(r)}{\sinh r}$ by Lemma \ref{lem: property g}, we have $(g'(r))^{2} - \frac{g^{2}(r)}{\sinh^{2} r} \leq 0$. Thus
\begin{align} \label{eqn: hyperbolic 2}
\sum_{i=1}^{n-1} \left(\frac{1}{\mu_{i}(\Omega)} - \frac{1}{\mu_{n}(\Omega)} \right) \int_{\Omega}\left((g'(r))^{2} - \frac{g^{2}(r)}{\sinh^{2} r}\right)  \frac{x_{i}^{2}}{r^{2}} dV \leq 0.
\end{align}
Substituting \eqref{eqn: hyperbolic 1} and \eqref{eqn: hyperbolic 2} in \eqref{eqn: sum1}, we have
\begin{align} \nonumber
\int_{\partial\Omega} g^{2}(r) dA  & \leq \sum_{i=1}^{n} \frac{1}{\mu_{i}(\Omega)} \int_{\Omega} \frac{g^{2}(r)}{\sinh^{2} r} dV + \frac{1}{\mu_{n}(\Omega)} \int_{\Omega}\left((g'(r))^{2} - \frac{g^{2}(r)}{\sinh^{2} r}\right) dV \\ \nonumber
& = \sum_{i=1}^{n-1} \frac{1}{\mu_{i}(\Omega)} \int_{\Omega} \frac{g^{2}(r)}{\sinh^{2} r} dV + \frac{1}{\mu_{n}(\Omega)} \int_{\Omega}(g'(r))^{2} dV \\ \nonumber
& \leq \sum_{i=1}^{n-1} \frac{1}{\mu_{i}(\Omega)} \int_{\Omega} \frac{g^{2}(r)}{\sinh^{2} r} dV + \sum_{i=1}^{n-1} \frac{1}{\mu_{i}(\Omega)} \int_{\Omega} \frac{(g'(r))^{2}}{(n-1)} dV \\ \label{eqn: integral hyperbolic 3}
& = \frac{1}{(n-1)} \sum_{i=1}^{n-1} \frac{1}{\mu_{i}(\Omega)} \int_{\Omega} \left(\frac{n-1}{\sinh^{2} r} g^{2}(r) + (g'(r))^{2}\right) dV.
\end{align}
Define $F(r) = \frac{n-1}{\sinh^{2} r} g^{2}(r) + (g'(r))^{2}$. Since $F(r)$ is a decreasing function of $r$,
\begin{align*}
\int_{\Omega} F(r) dV &= \int_{\Omega \cap B(R)} F(r) dV + \int_{\Omega\setminus (\Omega \cap B(R))} F(r) dV \\
& \leq \int_ {B(R)} F(r) dV - \int_{B(R) \setminus (\Omega \cap B(R))} F(r) dV + \int_{\Omega\setminus (\Omega \cap B(R))} F(R) dV \\
& \leq \int_ {B(R)} F(r) dV - \int_{B(R) \setminus (\Omega \cap B(R))} F(R) dV+ \int_{\Omega\setminus (\Omega \cap B(R))} F(R) dV \\
& = \int_ {B(R)} F(r) dV.
\end{align*}
Substituting this and $\int_{\partial\Omega} g^{2}(r) dA $ from Lemma \ref{lem: integral g} in \eqref{eqn: integral hyperbolic 3}, we get
\begin{align*}
g^{2}(R) \mbox{vol} \left(S(R)\right) & \leq \frac{1}{(n-1)} \sum_{i=1}^{n-1} \frac{1}{\mu_{i}(\Omega)} \int_{B(R)} \left(\frac{n-1}{\sinh^{2} r} g^{2}(r) + (g'(r))^{2}\right) dV.
\end{align*}
This gives
\begin{align*}
\frac{1}{(n-1)} \sum_{i=1}^{n-1} \frac{1}{\mu_{i}(\Omega)} \geq \frac{ \mbox{vol} \left(S(R)\right) g^{2}(R)}{\int_{B(R)} \left(\frac{n-1}{\sinh^{2} r} g^{2}(r) + (g'(r))^{2}\right) dV} = \frac{1}{\mu_{1}(B(R))}.
\end{align*}
Since $\mu_{1}(B(R))= \mu_{2}(B(R)) = \cdots = \mu_{n}(B(R))$, we have the desired inequality. Moreover, equality holds if and only if $\Omega$ is a geodesic ball.

\subsection{Proof of Theorem \ref{thm: Euclidean space}}
Since $g(r) = \frac{r}{n}$ and $\sin_{k} r = r$ for $M = \mathbb{R}^{n}$, Equation \eqref{eqn: sum1} can be written as
\begin{align*}
 \int_{\partial\Omega} r^{2} dA \leq \mbox{vol}(\Omega) \sum_{i=1}^{n} \frac{1}{\mu_{i}(\Omega)}.
\end{align*}
 By Lemma \ref{lem: integral g}, we get
 \begin{align*}
 R^{2} \mbox{vol}(S(R)) \leq \mbox{vol}(B(R)) \sum_{i=1}^{n} \frac{1}{\mu_{i}(\Omega)}.
 \end{align*}
 Since vol$(B(R))$ = $\frac{R}{n}$ vol$(S(R))$, thus the above inequality becomes
 \begin{align*}
 n R \leq \sum_{i=1}^{n} \frac{1}{\mu_{i}(\Omega)}.
 \end{align*}
 Recall that the first $n$ nonzero Steklov eigenvalues of $B(R)$ are $\mu_{1}(B(R)) = \mu_{2}(B(R)) = \cdots = \mu_{n}(B(R))= \frac{1}{R}$. Hence we have the desired result,
 \begin{align*}
 \sum_{i=1}^{n} \frac{1}{\mu_{i}(B(R))} \leq \sum_{i=1}^{n} \frac{1}{\mu_{i}(\Omega)}
 \end{align*}
and equality hold if and only if $\Omega = B(R)$.

\end{document}